\documentclass[preprint,12pt]{elsarticle}

\usepackage[utf8]{inputenc}
\usepackage[T1]{fontenc}
\usepackage{amsmath,amssymb,amsthm,mathtools}
\usepackage{enumitem}
\usepackage{array}
\usepackage{hyperref}

\hypersetup{hidelinks}

\theoremstyle{plain}
\newtheorem{theorem}{Theorem}[section]
\newtheorem{lemma}[theorem]{Lemma}
\newtheorem{proposition}[theorem]{Proposition}

\newtheorem{problem}[theorem]{Problem}

\theoremstyle{definition}
\newtheorem{definition}[theorem]{Definition}

\newtheorem{counterexample}[theorem]{Counterexample}

\theoremstyle{remark}
\newtheorem{remark}[theorem]{Remark}

\newcommand{\R}{\mathbb{R}}

\newcommand{\E}{\mathbb{E}}
\newcommand{\Prob}{\mathbb{P}}
\newcommand{\Lap}{\mathcal{L}}
\newcommand{\CM}{\mathrm{CM}}
\newcommand{\BF}{\mathrm{BF}}
\newcommand{\SBF}{\mathrm{SBF}}
\newcommand{\CBF}{\mathrm{CBF}}
\DeclareMathOperator{\Cov}{Cov}
\DeclareMathOperator{\Pois}{Poisson}
\DeclareMathOperator{\supp}{supp}

\makeatletter
\def\ps@pprintTitle{%
  \let\@oddhead\@empty
  \let\@evenhead\@empty
  \let\@oddfoot\@empty
  \let\@evenfoot\@empty}
\makeatother

\begin{document}

\begin{frontmatter}

\title{Bernstein Functions at Work:\\
Coalescents, Copulas, and Subordination}
\author[addr1]{Domingos S. P. Salazar\corref{cor1}}
\address[addr1]{Unidade de Educa\c{c}\~ao a Dist\^ancia e Tecnologia,
Universidade Federal Rural de Pernambuco,
52171-900 Recife, Pernambuco, Brazil}
\cortext[cor1]{Corresponding author.}

\begin{abstract}
Several positivity questions in stochastic processes, dependence modeling,
fractional analysis, and renewal theory reduce to a common recognition task:
after normalization, identify the object as a Laplace transform, a potential
density, an inverse-flow coefficient, or a finite kernel average, and then read
the sign pattern from that representation. We develop this recognition calculus
for completely monotone functions, Bernstein functions, special Bernstein
functions, and probabilistic realizations through subordinators and mixing
measures.

The main affirmative results settle three narrowly stated source questions in
the conventions used by their source papers. M\"ohle's
Problem~6.3 on the block-counting process of exchangeable coalescents with
residual singleton mass (dust) is proved by a finite-simplex ordered-pair kernel
certificate. For the
Pearse--Bondell power-divergence copula generators, we prove complete
monotonicity of the inverse throughout the remaining strict negative range
$\lambda\le-1$ identified in their Section~3.8. Together with the special cases already verified
in the source paper, this yields Archimedean copulas in every dimension for
$\lambda\le-1$. The Bendikov--Cygan monotonicity question for discrete renewal
sequences attached to special Bernstein functions is answered by representing
the potential kernel as a Gamma average of a nonincreasing density.

Supporting representation and boundary results cover Sibisi's
Prabhakar--Pollard $Q$-measure, the
Mecke--Nagel--Weiss atom at zero, and the cubic branch criterion $a^2\ge3b$.
\end{abstract}


\end{frontmatter}

\section{Introduction}

The classes of completely monotone (CM), Bernstein (BF), and Stieltjes
functions are the analytic shadow of three probabilistic operations: forming a
Laplace transform of a positive measure, forming the Laplace exponent of a
subordinator, and iterating the two. This dictionary---Bernstein--Widder for CM
functions, L\'evy--Khintchine for Bernstein functions, and Bochner
subordination for their composition---turns many analytic questions into
statements about a representing measure that one can often write down
explicitly \cite{schilling-song-vondracek,berg-forst}.

The present paper treats source problems from several literatures with separate
standard languages. For exchangeable coalescents and coalescents with
simultaneous multiple collisions we use the ranked-simplex framework of
Pitman, Sagitov, and Schweinsberg, with dust in the sense used in M\"ohle's
block-counting problem \cite{pitman-coalescents,sagitov,schweinsberg,mohle}.
For Archimedean copulas we use standard complete-monotonicity and
\(d\)-monotonicity criteria \cite{kimberling,nelsen,sklar,mcneil-neslehova}.
For Prabhakar and Mittag-Leffler functions we cite the classical sources of
Pollard and Feller, together with modern treatments
\cite{pollard,feller,prabhakar,gorenflo-kilbas-mainardi-rogozin,mainardi-garrappa,garra-garrappa,sibisi-prabhakar}.
For renewal sequences attached to Bernstein functions we use the discrete
subordination setting of Bendikov--Cygan and the potential-density theorem for
special Bernstein functions \cite{bendikov-cygan,schilling-song-vondracek}.
The common reduction is to Laplace-measure, inverse-flow, or finite-kernel
positivity.
Our organizing thesis is a recognition principle:

\begin{quote}
\emph{After the correct normalization, the special-function object at issue is a
moment, a survival function, or a subordination push-forward of a positive
measure; its analytic sign pattern is then determined by the support and
monotonicity of that measure.}
\end{quote}

The paper is organized by mechanism: representation, conjecture resolution,
Bernstein structure, and obstruction.

\begin{table}[t]
\centering
\caption{Main mechanisms and cited source problems.}
\label{tab:mechanisms}
\footnotesize
\begin{tabular}{@{}>{\raggedright\arraybackslash}p{0.39\linewidth}
>{\raggedright\arraybackslash}p{0.25\linewidth}
>{\raggedright\arraybackslash}p{0.26\linewidth}@{}}
\hline
Source object & Mechanism & Output \\
\hline
Prabhakar--Pollard \(Q\) \cite{prabhakar,pollard,sibisi-prabhakar} &
Stable subordination & \(Q\)-push-forward \\
Mecke--Nagel--Weiss \cite{mecke-nagel-weiss} & Exponential race &
Atom-at-zero criterion \\
M\"ohle block counts \cite{gaiser-mohle,mohle} & Collision kernel & Second-moment domination \\
Pearse--Bondell copulas \cite{pearse-bondell} & Inverse-flow series &
Negative-range CM \\
Li cubic branch \cite{li-relativistic} & Discriminant test & \(a^2\ge3b\) \\
Bendikov--Cygan renewal \cite{bendikov-cygan} & Gamma covariance &
\(C(k)\) decreasing \\
\hline
\end{tabular}
\end{table}

\emph{Canonical representations} (Section~\ref{sec:representations}). We show
that the Prabhakar \cite{prabhakar} three-parameter Mittag-Leffler
\cite{mittag-leffler,prabhakar} \(Q\)-measure in the strict Pollard range
\cite{pollard,sibisi-prabhakar} is an \(\alpha\)-stable subordination of a
transform-normalized Pollard measure (Theorem~\ref{thm:prabhakar}), and that
the Mecke--Nagel--Weiss Laplace-survival transform and its Poissonization
\cite{mecke-nagel-weiss} are realized by an exponential race, with the
atom-at-zero caveat needed for the literal finite-valued formulation
(Theorem~\ref{thm:mecke}).

\emph{Two affirmative source questions} (Section~\ref{sec:conjectures}). We prove
M\"ohle's Problem~6.3 for exchangeable coalescents with dust \cite{mohle}
(Theorem~\ref{thm:mohle}). Here dust means residual singleton mass. We also
prove complete monotonicity on the remaining strict negative range
\(\lambda\le-1\) identified by Pearse--Bondell for power-divergence copula
generators \cite[Section~3.8]{pearse-bondell}
(Theorem~\ref{thm:pd}). Both results address explicitly stated source
questions.

\emph{Bernstein structure} (Section~\ref{sec:bernstein}). We characterize the
Bernstein branch of a cubic inverse-polynomial by a discriminant inequality
(Theorem~\ref{thm:li}), as a model test case motivated by Li's
relativistic-diffusion subordination framework \cite{li-relativistic}, and
prove that the Bendikov--Cygan renewal sequence attached to a special
Bernstein function is nonincreasing \cite{bendikov-cygan}
(Theorem~\ref{thm:bendikov}).

\emph{Obstructions and certificate targets} (Section~\ref{sec:obstructions}).
Several examples mark the boundary of the calculus: a
Rastegar--Roitershtein redundancy conjecture fails
\cite{rastegar-roitershtein} (Counterexample~\ref{cx:rr}); a
complete-monotonicity reading of a Jonckheere--Shneer front-propagation equation
\cite{jonckheere-shneer} has non-CM solutions
(Counterexample~\ref{cx:js}); and Townes real-valued mixed-Poisson
\cite{townes-mixed,townes-stable} and Bazhlekova--Bazhlekov gap examples
\cite{bazhlekova} are retained only as certificate targets.

The recurring technical engines are few: Bernstein--Widder inversion; the
covariance sign of a monotone function against an increasing statistic; Fa\`a
di Bruno sign induction applied to an inverse ODE $y'=g(y)$ with $g\in\CM$; and
explicit certificates for the verified obstructions. We have isolated each as a
lemma so that the applications read uniformly.

\section{Preliminaries}\label{sec:prelim}

We recall the standard vocabulary; see \cite{schilling-song-vondracek} for a
complete treatment.

\begin{definition}[Completely monotone, Bernstein, and Stieltjes classes]\label{def:classes}
Let all functions below be defined on \((0,\infty)\) and take values in
\([0,\infty)\).
\begin{enumerate}[label=\textup{(\roman*)},leftmargin=2.3em]
\item A function \(f\) is \emph{completely monotone}, written \(f\in\CM\), if
\(f\in C^\infty\) and
\((-1)^n f^{(n)}(x)\ge0\) for every \(n\ge0\) and \(x>0\). It is
\emph{strictly completely monotone} if all these inequalities are strict.

\item A function \(g\) is a \emph{Bernstein function}, written \(g\in\BF\), if
\(g\in C^\infty\) and \(g'\in\CM\).

\item A Bernstein function \(g\) is a \emph{special Bernstein function}, written
\(g\in\SBF\), if \(x/g(x)\in\BF\).

\item A Bernstein function \(g\) is a \emph{complete Bernstein function}, written
\(g\in\CBF\), if its L\'evy measure has a completely monotone density.

\item A function \(h\) is a \emph{Stieltjes function} if
\begin{equation}
h(x)=\frac{a}{x}+b+\int_{(0,\infty)}\frac{1}{x+t}\,\sigma(dt),\qquad x>0,
\end{equation}
where \(a,b\ge0\) and
\(\int_{(0,\infty)}(1+t)^{-1}\,\sigma(dt)<\infty\).
\end{enumerate}
\end{definition}

\noindent
The two representation facts used next are standard in Bernstein-function
theory; we use the formulation in \cite{schilling-song-vondracek}.

\begin{proposition}[Standard Bernstein representations]\label{prop:standard-reps}
The following standard facts are used below.
\begin{enumerate}[label=\textup{(\roman*)},leftmargin=2.3em]
\item If \(g\in\BF\), then there are \(a,b\ge0\) and a measure \(\nu\) on
\((0,\infty)\) satisfying
\(\int_{(0,\infty)}(1\wedge t)\,\nu(dt)<\infty\) such that
\begin{equation}
g(x)=a+bx+\int_{(0,\infty)}(1-e^{-xt})\,\nu(dt).
\end{equation}

\item If \(\psi\in\SBF\) and \(\psi(0)=0\), then the potential measure with
Laplace transform \(1/\psi\) admits a representation
\begin{equation}
\frac1{\psi(\lambda)}
=b+\int_0^\infty e^{-\lambda t}u(t)\,dt,\qquad \lambda>0,
\end{equation}
where \(b\ge0\) and \(u\) has a nonincreasing locally integrable version.
\end{enumerate}
\end{proposition}

\noindent
The Laplace-transform characterization of completely monotone functions used
below is the Bernstein--Widder theorem; we use Widder's classical form and the
modern Bernstein-function reference \cite{widder,schilling-song-vondracek}.

\begin{theorem}[Bernstein--Widder]\label{thm:bw}
\(f\in\CM\) if and only if there is a unique positive Radon measure \(\mu\) on
\([0,\infty)\) such that
\begin{equation}
f(x)=\int_{[0,\infty)}e^{-xt}\,d\mu(t).
\end{equation}
The representing measure satisfies
\begin{equation}
\int_{[0,\infty)}e^{-xt}\,d\mu(t)<\infty,\qquad x>0.
\end{equation}
If \(\mu\) is finite, then \(f(0^+)=\mu([0,\infty))\) and
\(f(\infty)=\mu(\{0\})\). Hence \(f\) is a normalized survival transform
(\(f(0^+)=1\), \(f(\infty)=0\)) exactly when \(\mu\) is a probability measure
with no atom at \(0\).
\end{theorem}

We use three elementary lemmas repeatedly.

\begin{lemma}[Positive-kernel principle]\label{lem:poskernel}
Let $d\mu(t)\ge0$ and suppose
\begin{equation}
f(x)=\int_0^\infty e^{-xt}\,d\mu(t)
\end{equation}
converges on $(0,\infty)$. Then $f\in\CM$. If \(\mu\ne0\) and
\(\mu((0,\infty))>0\), then \(f\) is strictly completely monotone.
\end{lemma}
\begin{proof}
Differentiation under the integral sign gives
\begin{equation}
(-1)^n f^{(n)}(x)=\int_0^\infty t^n e^{-xt}\,d\mu(t)\ge0.
\end{equation}
For \(n=0\) the integral is positive when \(\mu\ne0\). For \(n\ge1\), the
integral is positive exactly when \(\mu((0,\infty))>0\). Thus the only nonzero
constant obstruction to strictness is a measure supported at \(0\).
\end{proof}

\begin{lemma}[Covariance sign for monotone rearrangements]\label{lem:cov}
Let $\nu$ be a probability measure on $\R$, let $\phi$ be nondecreasing and
$\psi$ nonincreasing on $\supp\nu$, and let $T\sim\nu$. Then
$\Cov(\phi(T),\psi(T))\le0$, provided the moments exist.
\end{lemma}
\begin{proof}
With $T,T'$ i.i.d.\ $\sim\nu$,
\begin{equation}
\Cov(\phi(T),\psi(T))=\tfrac12\,\E\!\left[(\phi(T)-\phi(T'))(\psi(T)-\psi(T'))\right].
\end{equation}
The integrand is $\le0$ pointwise, since $\phi$ and $\psi$ move in opposite
directions along $T-T'$.
\end{proof}

\begin{lemma}[Inverse-ODE sign induction]\label{lem:invode}
Let $g\in\CM$ on an interval $I$, and let $y\colon(0,\infty)\to I$ be a positive
$C^\infty$ solution of $y'=g(y)$ with $y'>0$. Then $y\in\BF$; that is,
$(-1)^{n}y^{(n+1)}\ge0$ for all $n\ge0$. If $g$ is strictly completely monotone,
the inequalities are strict.
\end{lemma}
\begin{proof}
We show $(-1)^{n-1}y^{(n)}\ge0$ for $n\ge1$ by induction. The base case $y'>0$
is the hypothesis. Assume it holds for $1\le j\le n$. Differentiating
$y^{(n+1)}=(g\circ y)^{(n)}$ by Fa\`a di Bruno's formula \cite{comtet}
expresses $y^{(n+1)}$ as a sum of
terms
\begin{equation}
g^{(k)}(y)\prod_{i} y^{(j_i)},\qquad \sum_i j_i = n,\ \ k=\#\{i\},
\end{equation}
each with a nonnegative combinatorial coefficient. The factor $g^{(k)}(y)$ has
sign $(-1)^k$, and $\prod_i y^{(j_i)}$ has sign $\prod_i(-1)^{j_i-1}
=(-1)^{n-k}$. The product has sign $(-1)^k(-1)^{n-k}=(-1)^n$, so
$(-1)^n y^{(n+1)}\ge0$. Strictness propagates when
\((-1)^k g^{(k)}>0\) on \(I\) for every \(k\ge0\) and the preceding induction
inequalities are strict.
\end{proof}

\section{Canonical representations}\label{sec:representations}

\subsection{The Prabhakar \texorpdfstring{$Q$}{Q}-measure as a stable subordination}

Let $0<\alpha<1$, $\eta>0$, $\beta>\alpha\eta$, and recall the Prabhakar
(three-parameter Mittag-Leffler) function \cite{mittag-leffler,prabhakar}
\begin{equation}
E_{\alpha,\beta}^{\eta}(z)=\sum_{m\ge0}\frac{(\eta)_m\,z^m}{m!\,\Gamma(\alpha
m+\beta)} .
\end{equation}
Here \((a)_m\) is the rising Pochhammer symbol, with \((a)_0=1\) and
\((a)_m=\Gamma(a+m)/\Gamma(a)\).
Sibisi \cite{sibisi-prabhakar}, building on Pollard's complete-monotonicity
theorem \cite{pollard} and the Feller stable-law proof \cite{feller},
introduced a transform-normalized Pollard measure
$P_{\alpha,\beta}^{\eta}$ characterized in the strict Pollard range by
\(\int_0^\infty e^{-xr}\,dP_{\alpha,\beta}^{\eta}(r)=E_{\alpha,\beta}^{\eta}(-x)\).
The complete monotonicity and integral representations of Prabhakar functions
have substantial prior literature
\cite{gorenflo-kilbas-mainardi-rogozin,mainardi-garrappa,garra-garrappa}.
The point here is narrower: we identify the associated \(Q\)-measure as the
stable-subordination push-forward of \(P_{\alpha,\beta}^{\eta}\).
We use Sibisi's unnormalized Prabhakar convention
\begin{equation}
E_{\alpha,\beta}^{\eta}(0)=\frac1{\Gamma(\beta)};
\end{equation}
hence the representing object below is a finite measure of mass
\(1/\Gamma(\beta)\), and multiplying by \(\Gamma(\beta)\) gives the associated
probability law.

\begin{lemma}[Stable-subordination normal form]\label{lem:stablesub}
Let $P$ be a finite positive measure on $[0,\infty)$ with Laplace transform
$F(u)=\int_0^\infty e^{-ur}\,dP(r)$, let $0<\alpha<1$ and $\lambda>0$, and let
$S_\alpha$ be the positive $\alpha$-stable variable with
$\E e^{-sS_\alpha}=e^{-s^\alpha}$. Define the finite measure
\begin{equation}
Q(A)=\int_{[0,\infty)}\Prob\{(\lambda r)^{1/\alpha}S_\alpha\in A\}\,dP(r).
\end{equation}
Then $\int_0^\infty e^{-xt}\,dQ(t)=F(\lambda x^\alpha)$ for all $x>0$.
\end{lemma}
\begin{proof}
By Tonelli's theorem and the stable Laplace identity,
\begin{equation}
\int_0^\infty e^{-xt}\,dQ(t)
=\int_{[0,\infty)}\E e^{-x(\lambda r)^{1/\alpha}S_\alpha}\,dP(r)
=\int_{[0,\infty)}e^{-\lambda r x^\alpha}\,dP(r)
=F(\lambda x^\alpha). \qedhere
\end{equation}
\end{proof}

\begin{theorem}[Prabhakar $Q$-measure by stable subordination]\label{thm:prabhakar}
Let $0<\alpha<1$, $\eta>0$, $\beta>\alpha\eta$, $\lambda>0$, and define
\begin{equation}
Q_{\alpha,\beta}^{\eta}(A\mid\lambda)
=\int_{[0,\infty)}\Prob\{(\lambda r)^{1/\alpha}S_\alpha\in A\}\,
dP_{\alpha,\beta}^{\eta}(r).
\end{equation}
Then
\begin{equation}
\int_0^\infty e^{-xt}\,dQ_{\alpha,\beta}^{\eta}(t\mid\lambda)
=E_{\alpha,\beta}^{\eta}(-\lambda x^\alpha),\qquad x>0,
\end{equation}
and by uniqueness of finite Laplace transforms this is the source
$Q$-measure described above.
Its total mass is $1/\Gamma(\beta)$; the normalized law is
$\widehat Q_{\alpha,\beta}^{\eta}=\Gamma(\beta)\,Q_{\alpha,\beta}^{\eta}$.
\end{theorem}
\begin{proof}
Apply Lemma~\ref{lem:stablesub} with $P=P_{\alpha,\beta}^{\eta}$ and
$F(u)=E_{\alpha,\beta}^{\eta}(-u)$; the range condition $\beta>\alpha\eta$ is
exactly Sibisi's strict Pollard range in which $P_{\alpha,\beta}^{\eta}$ is a
positive measure. Uniqueness of finite Laplace transforms identifies the result
with $Q_{\alpha,\beta}^{\eta}(\cdot\mid\lambda)$. Since
$E_{\alpha,\beta}^{\eta}(0)=1/\Gamma(\beta)$, the representing measures have
total mass $1/\Gamma(\beta)$. If $S_\alpha$ has density $f_\alpha$ and
$P_{\alpha,\beta}^{\eta}$ has Pollard density $p_{\alpha,\beta}^{\eta}$, the
absolutely continuous part of $Q$ is
\begin{equation}
q_{\alpha,\beta}^{\eta}(t\mid\lambda)
=\int_0^\infty (\lambda r)^{-1/\alpha}
f_\alpha\!\left(\frac{t}{(\lambda r)^{1/\alpha}}\right)
p_{\alpha,\beta}^{\eta}(r)\,dr,
\end{equation}
with an atom $P_{\alpha,\beta}^{\eta}(\{0\})\delta_0$ if present. The boundary
$\beta=\alpha\eta$ is excluded, as Sibisi's ordinary-convolution density uses
$1/\Gamma(\beta-\alpha\eta)$; its limiting treatment is separate.
\end{proof}

\begin{remark}[Source problem]
Theorem~\ref{thm:prabhakar} gives a stable-subordination representation for
the \(Q\)-measure associated with Sibisi's Prabhakar--Pollard construction
\cite{sibisi-prabhakar}. Analytically, this is Bochner subordination:
\(x\mapsto x^\alpha\) is the Laplace exponent of \(S_\alpha\), so replacing
\(u\) by \(\lambda x^\alpha\) in a Laplace transform is the corresponding
push-forward under the stable subordinator. The proof uses the stable Laplace
identity as its only special-function input.
\end{remark}

\subsection{Laplace-survival and Poissonization}

Mecke's posthumously compiled manuscripts \cite{mecke-nagel-weiss} pose two
construction problems: realize a random variable whose distribution function is
$1-L_\zeta$, and realize a count whose probability generating function is a
Laplace transform of $\zeta$. We use the standard exponential-interarrival
construction of a Poisson process \cite{kingman-poisson}.

\begin{theorem}[Laplace-survival and Poissonization]\label{thm:mecke}
Let $\zeta\ge0$ with Laplace transform $L_\zeta(x)=\E e^{-x\zeta}$, and let $U$
be uniform on $(0,1)$ independent of $\zeta$.
\begin{enumerate}
\item[(i)] If $\Prob(\zeta>0)=1$, then $\xi=(-\log U)/\zeta$ satisfies
$F_\xi(x)=1-L_\zeta(x)$ for all $x\ge0$. For arbitrary $\zeta\ge0$ the same
holds on the extended half-line with $\xi=\infty$ on $\{\zeta=0\}$, and
\emph{no} finite-valued $\xi$ can satisfy the literal identity for all $x\ge0$
when $\Prob(\zeta=0)>0$.
\item[(ii)] If $\kappa_t\mid\zeta\sim\Pois(t\zeta)$, then
$\E[x^{\kappa_t}]=L_\zeta(t(1-x))$ for $0\le x\le1$, and $\kappa_t$ is realizable
from i.i.d.\ uniforms via the exponential-interarrival construction of a
Poisson process.
\end{enumerate}
\end{theorem}
\begin{proof}
Since $U$ is uniform, $E=-\log U$ is unit exponential. Conditioning on
$\zeta=z>0$, $\Prob(E/z\le x)=1-e^{-xz}$; on $\{\zeta=0\}$ the extended
convention gives $\xi=\infty$, hence $\Prob(\xi\le x\mid\zeta=0)=0=1-e^{-x\cdot0}$.
Taking expectations,
\begin{equation}
\Prob(\xi\le x)=\E[1-e^{-x\zeta}]=1-L_\zeta(x).
\end{equation}
For the impossibility, bounded convergence gives $L_\zeta(x)\to\Prob(\zeta=0)$
as $x\to\infty$, so $1-L_\zeta(x)\to1-\Prob(\zeta=0)<1$; a finite distribution
function must tend to $1$, so no finite-valued $\xi$ works when
$\Prob(\zeta=0)>0$.

For (ii), conditioning on $\zeta$,
\begin{equation}
\E[x^{\kappa_t}\mid\zeta]=\exp(t\zeta(x-1))=\exp(-t(1-x)\zeta),
\end{equation}
so $\E[x^{\kappa_t}]=\E e^{-t(1-x)\zeta}=L_\zeta(t(1-x))$. Taking i.i.d.\
uniforms $(\eta_i)$, $E_i=-\log\eta_i$, partial sums $T_n=E_1+\cdots+E_n$, and
$N(a)=\max\{n\ge0:T_n\le a\}$ realizes a standard Poisson process, and
$\kappa_t=N(t\zeta)$ is the desired conditional draw.
\end{proof}

\begin{remark}[Source problem]
The inverse direction is Bernstein--Widder (Theorem~\ref{thm:bw}): normalized
completely monotone survival functions are exactly Laplace transforms of
probability measures on $[0,\infty)$. Theorem~\ref{thm:mecke} realizes the two
Mecke--Nagel--Weiss construction questions \cite{mecke-nagel-weiss}; part (i)
gives a finite-valued realization exactly when \(\Prob(\zeta>0)=1\), and an
extended-valued realization otherwise. The atom-at-zero caveat is forced by
\(L_\zeta(\infty)=\Prob(\zeta=0)\).
\end{remark}

\section{Two affirmative conjectures}\label{sec:conjectures}

\subsection{M\"ohle's Problem 6.3 for coalescents with dust}

Let $\Delta=\{u=(u_1\ge u_2\ge\cdots\ge0):|u|=\sum_i u_i\le1\}$ be the
ranked-simplex parameter space used for \(\Xi\)-coalescents with simultaneous
multiple collisions \cite{pitman-coalescents,sagitov,schweinsberg} and for
M\"ohle's exchangeable coalescent with dust problem \cite{gaiser-mohle,mohle}:
the coordinates \(u_i\) are non-singleton family frequencies, while
\(u_0=1-|u|\) is the residual dust mass carried by singleton blocks. Given
\(u\in\Delta\), let
\((X_0,X_1,\ldots)\) have the infinite multinomial law obtained by assigning
\(n\) independent balls to boxes with probabilities
\((u_0,u_1,u_2,\ldots)\). Define
\begin{equation}
Y(n,u)=X_0+\sum_{i\ge1}\mathbf 1_{\{X_i>0\}},\qquad
\widetilde Y(n,u)=\frac{Y(n,u)}{n}-(1-|u|).
\end{equation}
Set
\begin{equation}
p_n(u)=\E\widetilde Y(n,u)-\E\bigl(\widetilde Y(n,u)^2\bigr).
\end{equation}
M\"ohle \cite{mohle} asks (Problem~6.3) whether \(p_n(u)\ge0\) for all
\(u\in\Delta\) and all integers \(n\ge1\), i.e.\ whether the second moment never
exceeds the first.

We first record the two exact integral identities that carry the proof.

\begin{lemma}[Binomial collision kernel]\label{lem:collision}
For every integer $n\ge2$ and $x,y\ge0$ with $x+y\le1$,
\begin{equation}
(1-x)^n+(1-y)^n-1-(1-x-y)^n
=-\,n(n-1)\int_0^x\!\!\int_0^y(1-r-t)^{n-2}\,dr\,dt,
\end{equation}
and $\;1-(1-x)^n=n\int_0^x(1-t)^{n-1}\,dt$.
\end{lemma}
\begin{proof}
Both are exact: differentiate each side in $x$ (and $y$) and match at $x=0$
(resp.\ $y=0$). For the first, $\partial_x\partial_y$ of the left side is
$-n(n-1)(1-x-y)^{n-2}$, matching the integrand, and both sides vanish when
$x=0$ or $y=0$. The second is the fundamental theorem of calculus applied to
$\tfrac{d}{dt}(1-(1-t)^n)=n(1-t)^{n-1}$.
\end{proof}

\begin{theorem}[M\"ohle Problem 6.3 is affirmative]\label{thm:mohle}
For every integer $n\ge1$ and every $u\in\Delta$, $p_n(u)\ge0$; equivalently
$\E(\widetilde Y(n,u)^2)\le\E\widetilde Y(n,u)$ on the full source simplex.
\end{theorem}
\begin{proof}
The case $n=1$ is recorded in M\"ohle's calculation \cite{mohle}:
$p_1(u)=|u|(1-|u|)\ge0$. Fix $n\ge2$ and first assume $u$ has finite support.
With $s=|u|$, M\"ohle's Lemma~6.1 and Problem~6.3 \cite{mohle} give the
ordered-pair formula
\begin{equation}
\begin{aligned}
p_n(u)
&=\frac{n-1}{n^2}A-\frac{s(1-s)}{n}+\frac{2(1-s)}{n}B\\
&\quad+\frac1{n^2}\sum_{i\ne j}
\Bigl((1-u_i)^n+(1-u_j)^n-1-(1-u_i-u_j)^n\Bigr),
\end{aligned}
\end{equation}
where $A=\sum_i(1-(1-u_i)^n)$, $B=\sum_i u_i(1-u_i)^{n-1}$, and the pair sum is
ordered. Applying the two identities of Lemma~\ref{lem:collision} and writing
\begin{equation}
\begin{aligned}
S&=\sum_i\int_0^{u_i}(1-t)^{n-1}\,dt,\\
T&=\sum_{i\ne j}\int_0^{u_i}\!\!\int_0^{u_j}(1-t-r)^{n-2}\,dr\,dt,
\end{aligned}
\end{equation}
the formula collapses to
\begin{equation}\label{eq:mohle-collapsed}
p_n(u)=\frac{n-1}{n}\,(S-T)+\frac{1-s}{n}\,(2B-s).
\end{equation}

\emph{Ordered-pair domination.} For fixed $i$ and $0\le t\le u_i$,
\begin{equation}
\sum_{j\ne i}\int_0^{u_j}(1-t-r)^{n-2}\,dr\le (s-u_i)(1-t)^{n-2},
\end{equation}
since \(n\) is an integer and \(0\le t+r\le u_i+u_j\le s\le1\) keeps the
integrand nonnegative, while $r\ge0$ gives
$1-t-r\le1-t$. Hence
\begin{equation}
\begin{aligned}
S-T
&\ge\sum_i\int_0^{u_i}
\bigl((1-t)^{n-1}-(s-u_i)(1-t)^{n-2}\bigr)\,dt\\
&=\sum_i\int_0^{u_i}
(1-t)^{n-2}\bigl((1-s)+(u_i-t)\bigr)\,dt\\
&\ge\frac{1-s}{n-1}\sum_i\bigl(1-(1-u_i)^{n-1}\bigr),
\end{aligned}
\end{equation}
using $\int_0^{u_i}(1-t)^{n-2}\,dt=\tfrac1{n-1}(1-(1-u_i)^{n-1})$ and discarding
the nonnegative $(u_i-t)$ term. Substituting into \eqref{eq:mohle-collapsed},
\begin{equation}
p_n(u)\ge\frac{1-s}{n}\sum_i
\Bigl(1-(1-u_i)^{n-1}+2u_i(1-u_i)^{n-1}-u_i\Bigr).
\end{equation}

\emph{Termwise positivity.} Put $q=1-u_i\in[0,1]$. The bracket equals
\begin{equation}
q+q^{n-1}-2q^{n}=q\bigl(1-q^{n-2}(2q-1)\bigr).
\end{equation}
If $q\le\tfrac12$ the parenthesized factor is $\ge1$; if $q\ge\tfrac12$ then
$0\le q^{n-2}(2q-1)\le1$. Either way each term is $\ge0$, and $1-s\ge0$, so
$p_n(u)\ge0$ for finitely supported $u$.

\emph{Countable support.} Let \(u^{(m)}=(u_1,\ldots,u_m,0,\ldots)\) and
\(s_m=|u^{(m)}|\). Then \(s_m\uparrow s\). The one-body terms in the displayed
formula converge by dominated convergence:
\begin{equation}
0\le 1-(1-u_i)^n\le n u_i,\qquad
0\le u_i(1-u_i)^{n-1}\le u_i,\qquad \sum_i u_i=s<\infty .
\end{equation}
For the ordered-pair term set
\begin{equation}
D_n(x,y)=(1-x)^n+(1-y)^n-1-(1-x-y)^n,\qquad x,y\ge0,\ x+y\le1 .
\end{equation}
Lemma~\ref{lem:collision} gives
\begin{equation}
D_n(x,y)=-n(n-1)\int_0^x\!\!\int_0^y(1-r-t)^{n-2}\,dr\,dt,
\end{equation}
hence
\begin{equation}
|D_n(x,y)|\le n(n-1)xy .
\end{equation}
Therefore
\begin{equation}
\sum_{i\ne j}|D_n(u_i,u_j)|
 \le n(n-1)\sum_{i\ne j}u_i u_j
 \le n(n-1)s^2<\infty .
\end{equation}
Since \(D_n(u_i^{(m)},u_j^{(m)})\to D_n(u_i,u_j)\) for every ordered pair
\((i,j)\), dominated convergence for counting measure on
\(\{(i,j):i\ne j\}\) yields convergence of the ordered-pair sums. Thus the
right-hand side of M\"ohle's formula satisfies \(p_n(u^{(m)})\to p_n(u)\), and
the finite-support inequality \(p_n(u^{(m)})\ge0\) passes to the limit.
\end{proof}

\begin{remark}
Theorem~\ref{thm:mohle} proves the nonnegativity assertion in
M\"ohle's Problem~6.3 \cite{mohle}. The reusable object is a
\emph{finite-simplex ordered-pair kernel certificate}:
negative pair-collision terms of the form $(1-x)^n+(1-y)^n-1-(1-x-y)^n$ are
dominated by one-body integrals precisely because the simplex constraint
$s\le1$ leaves a residual factor $1-s$.
\end{remark}

\subsection{The remaining Pearse--Bondell complete-monotonicity conjecture}

A standard Archimedean-generator criterion identifies the all-dimensional case
with complete monotonicity of the inverse
\cite{kimberling,nelsen,mcneil-neslehova}.

\begin{proposition}[All-dimensional Archimedean-generator criterion]\label{prop:arch-criterion}
Let \(\phi\colon(0,1]\to[0,\infty)\) be continuous, strictly decreasing, and
normalized by \(\phi(1)=0\), and suppose that its strict inverse
\(\phi^{-1}\colon(0,\infty)\to(0,1)\) is completely monotone. Then \(\phi\) is
an Archimedean generator in every dimension.
\end{proposition}
\begin{proof}
This is the standard Kimberling criterion in strict-generator form
\cite{kimberling,nelsen,mcneil-neslehova}. By Theorem~\ref{thm:bw},
\(\phi^{-1}\) is the Laplace transform of a probability measure \(M\) on
\([0,\infty)\). For each \(d\ge2\),
\begin{equation}
\begin{aligned}
C(u_1,\ldots,u_d)
&=\phi^{-1}\!\left(\sum_{i=1}^d\phi(u_i)\right)\\
&=\int_{[0,\infty)}\prod_{i=1}^d e^{-s\phi(u_i)}\,M(ds).
\end{aligned}
\end{equation}
For fixed \(s\), the map \(u\mapsto e^{-s\phi(u)}\) is a distribution function
on \([0,1]\). The displayed mixture has uniform one-dimensional margins because
\begin{equation}
\int_{[0,\infty)} e^{-s\phi(u)}\,M(ds)=\phi^{-1}(\phi(u))=u.
\end{equation}
Hence \(C\) is a \(d\)-copula for every \(d\).
\end{proof}

Pearse and Bondell \cite{pearse-bondell} introduce the power-divergence
generator $\phi_\lambda$ and prove complete monotonicity of
$\phi_\lambda^{-1}$ in several cases. Their Section~3.8 leaves the strict
negative range \(\lambda\le-1\) as a complete-monotonicity conjecture. The
divergence family belongs to the Cressie--Read power-divergence family and the
broader \(\phi\)-divergence tradition
\cite{csiszar,morimoto,ali-silvey,cressie-read,read-cressie,amari}.
We use their decreasing-generator convention
\(\phi_\lambda\colon(0,1]\to[0,\infty)\), normalized by
\(\phi_\lambda(1)=0\), and write its strict inverse as
\(\phi_\lambda^{-1}\).
For the negative range used below, write \(\lambda=-\gamma\) with
\(\gamma\ge1\). The generator is
\begin{equation}\label{eq:pd-generator}
\phi_{-\gamma}(x)=
\begin{cases}
\dfrac{x^{1-\gamma}+(\gamma-1)x-\gamma}{\gamma(\gamma-1)},
& \gamma>1,\\[0.9em]
x-1-\log x, & \gamma=1,
\end{cases}
\qquad 0<x\le1.
\end{equation}
It satisfies \(\phi_{-\gamma}(1)=0\) and
\(\phi_{-\gamma}(0^+)=+\infty\).

\begin{theorem}[Power-divergence inverse is completely monotone]\label{thm:pd}
For every $\lambda\le-1$ the strict inverse $\phi_\lambda^{-1}$ of the
power-divergence generator in \eqref{eq:pd-generator} is strictly completely
monotone on $(0,\infty)$, extends continuously to $0$ with value \(1\), and
satisfies \(\lim_{t\to\infty}\phi_\lambda^{-1}(t)=0\). Consequently, by
Proposition~\ref{prop:arch-criterion}, \(\phi_\lambda\) is a valid
Archimedean generator in every dimension for all \(\lambda\le-1\).
\end{theorem}
\begin{proof}
Write $\lambda=-\gamma$ with $\gamma\ge1$. From \eqref{eq:pd-generator}, in both
cases
\begin{equation}
\phi_{-\gamma}'(x)=\frac{1-x^{-\gamma}}{\gamma}=-\frac{1-x^\gamma}{\gamma x^\gamma}<0
\quad(0<x<1),
\end{equation}
with $\phi_{-\gamma}(1)=0$ and $\phi_{-\gamma}(0^+)=+\infty$. Thus
$u(t)=\phi_{-\gamma}^{-1}(t)$ is a strict $C^\infty$ bijection
$(0,\infty)\to(0,1)$, and by the inverse-function theorem
\begin{equation}
u'(t)=-\frac{\gamma\,u(t)^\gamma}{1-u(t)^\gamma}.
\end{equation}
Set
\begin{equation}
H_\gamma(x)=\frac{\gamma x^\gamma}{1-x^\gamma}=\gamma\sum_{m\ge1}x^{m\gamma}
\quad(0<x<1),\qquad L_\gamma=H_\gamma(x)\frac{d}{dx},
\end{equation}
so that $\frac{d}{dt}F(u(t))=-(L_\gamma F)(u(t))$ for smooth $F$. With
$P_0(x)=x$ and $P_{n+1}=L_\gamma P_n$, induction gives
\begin{equation}
(-1)^n u^{(n)}(t)=P_n(u(t)),\qquad n\ge0.
\end{equation}

It remains to prove \(P_n>0\) on \((0,1)\). For a monomial \(x^\beta\) with
\(\beta>0\),
\begin{equation}
L_\gamma x^\beta
=\gamma\beta\sum_{m\ge1}x^{\beta-1+m\gamma}.
\end{equation}
Iterating this identity gives, for \(n\ge1\),
\begin{equation}
\begin{aligned}
P_n(x)
&=\gamma^n\sum_{m_1,\dots,m_n\ge1}
\Bigl(\prod_{j=0}^{n-1}\beta_j\Bigr)x^{\beta_n},\\
\beta_0&=1,\qquad
\beta_j=1-j+\gamma(m_1+\cdots+m_j).
\end{aligned}
\end{equation}
For \(1\le j\le n\),
\begin{equation}
\beta_j\ge 1-j+\gamma j=1+j(\gamma-1).
\end{equation}
When \(\gamma=1\), this lower bound equals \(1\); when \(\gamma>1\), it is
strictly larger than \(1\). Thus every displayed coefficient and exponent is
positive. The series, and the one derivative needed to apply \(L_\gamma\)
termwise at the next step, are locally uniformly convergent on \((0,1)\): if
\(0<x\le r<1\) and
\(M=m_1+\cdots+m_n\), then
\begin{equation}
\prod_{j=0}^{n-1}\beta_j \le C_{n,\gamma}(1+M)^n,\qquad
x^{\beta_n}\le r^{1-n+\gamma M},
\end{equation}
and the number of \(n\)-tuples with sum \(M\) is \(O(M^{n-1})\). Hence the
absolute series is bounded by a constant times
\begin{equation}
\sum_{M\ge n}(1+M)^{2n-1} r^{\gamma M},
\end{equation}
which converges. Differentiating one term multiplies it by at most a further
constant times \((1+M)r^{-1}\) on \(0<x\le r<1\), so the differentiated series
has the same polynomial-times-geometric convergence form. Therefore \(P_n(x)>0\)
for every \(x\in(0,1)\). Since
\(P_0(x)=x>0\), the identity
\begin{equation}
(-1)^n u^{(n)}(t)=P_n(u(t))
\end{equation}
gives strict complete monotonicity of \(u=\phi_\lambda^{-1}\) on
\((0,\infty)\). Finally, \(\phi_{-\gamma}(1)=0\) and
\(\phi_{-\gamma}(0^+)=+\infty\), so \(u(t)\uparrow1\) as \(t\downarrow0\) and
\(u(t)\downarrow0\) as \(t\to\infty\).
\end{proof}

\begin{remark}[Source problem]
Theorem~\ref{thm:pd} proves complete monotonicity on the remaining strict
negative range \(\lambda\le-1\) identified in Pearse--Bondell's discussion of
power-divergence copula generators \cite[Section~3.8]{pearse-bondell}. Combined
with the special cases verified in the source paper, it gives the full
conclusion for \(\lambda\le-1\). Under the
Kimberling--Nelsen criterion \cite{kimberling,nelsen}, complete monotonicity of
\(\phi_\lambda^{-1}\) gives Archimedean copulas in every dimension; the
finite-dimensional \(d\)-monotone formulation is due to
McNeil--Ne\v{s}lehov\'a \cite{mcneil-neslehova}.

The proof is a positive generalized power-series flow for the decreasing
inverse equation \(u'=-H_\gamma(u)\). The key point is the explicit positivity
of all intermediate exponents \(\beta_j\), which proves the complete
monotonicity conjectured by Pearse and Bondell for the remaining range
\(\lambda\le-1\).
\end{remark}

\section{Bernstein structure}\label{sec:bernstein}

\subsection{A cubic inverse-polynomial Bernstein branch}

Li's generalized dual relativistic-diffusion framework leads to inverse-symbol
Bernstein checks for subordination constructions \cite{li-relativistic}. The
cubic inverse-polynomial model below gives a clean test case: the inverse branch
is a Bernstein function exactly in the discriminant regime.

\begin{theorem}[Cubic discriminant criterion]\label{thm:li}
Let $a,b\ge0$ and let $\varphi$ be the positive inverse of
\begin{equation}
\lambda=\varphi^3+a\varphi^2+b\varphi.
\end{equation}
Then $\varphi\in\BF$ on $(0,\infty)$ if and only if $a^2\ge3b$.
\end{theorem}
\begin{proof}
\emph{Sufficiency.} Assume $\Delta:=a^2-3b\ge0$. Since $b\ge0$ we have
$\sqrt\Delta\le a$, so with $r_\pm=(a\pm\sqrt\Delta)/3$ one has $0\le r_-\le r_+$
and
\begin{equation}
P'(x)=3x^2+2ax+b=3(x+r_-)(x+r_+).
\end{equation}
Thus $g(x)=1/P'(x)$ is a product of the completely monotone factors
$(x+r_\pm)^{-1}$ (with the repeated-root and $r=0$ cases handled by the same
formula or a limit, e.g.\ $1/(3x^2)\in\CM$), hence $g\in\CM$. The inverse obeys
$\varphi'(\lambda)=g(\varphi(\lambda))$, and Lemma~\ref{lem:invode} gives
$\varphi'\in\CM$, i.e.\ $\varphi\in\BF$.

\emph{Necessity.} Assume $a^2<3b$; then $b>0$ and
$P'(x)=3x^2+2ax+b>0$ for all real $x$. Hence \(P\colon\R\to\R\) is a strictly
increasing bijection. For every real \(\lambda\), the real inverse
\(x=P^{-1}(\lambda)\) satisfies \(P'(x)>0\), so the complex implicit-function
theorem gives a holomorphic inverse branch in a neighborhood of \(\lambda\).
These local branches agree on overlaps by uniqueness, and therefore the branch
continued from \(0\) is holomorphic along every compact interval in the
negative real axis. In particular, it has no singularity at any negative real
point. Because \(b>0\), \(\varphi\) is analytic at \(0\):
\(\varphi(\lambda)=\sum_{n\ge1}c_n\lambda^n\), with finite radius \(R\). Indeed,
if \(R=\infty\), then \(P(\varphi)=\lambda\) and \(|P(w)|\ge|w|^3/2\) for large
\(|w|\) would force \(|\varphi(\lambda)|\le C(1+|\lambda|)^{1/3}\). By Cauchy's
estimates such an entire function has zero derivative, contradicting
\(P'(\varphi)\varphi'=1\).

Suppose $\varphi\in\BF$. Then $\varphi'\in\CM$ and, $\varphi$ being analytic at
$0$, $(-1)^{n-1}c_n\ge0$ for $n\ge1$. Hence
$F(z):=-\varphi(-z)=\sum_{n\ge1}(-1)^{n-1}c_n z^n$ has nonnegative coefficients
and finite radius $R$. The case \(R=\infty\), including the possibility that
\(F\) is a polynomial, has already been excluded by the preceding entire-growth
argument. Hence Pringsheim's theorem \cite{pringsheim} implies that \(F\) has a
singularity at \(z=R\),
i.e.\ $\varphi$ is singular at $\lambda=-R<0$. This contradicts the negative-axis
analyticity above. Therefore $\varphi\notin\BF$ when $a^2<3b$.
\end{proof}

\begin{remark}[Source problem]
Theorem~\ref{thm:li} gives the Bernstein-function regime for the positive
inverse branch of the cubic inverse-polynomial motivated by Li's generalized
dual relativistic-diffusion subordination setting \cite{li-relativistic}. The
criterion is the discriminant inequality \(a^2\ge3b\).
\end{remark}

\subsection{Special Bernstein functions have decreasing renewal sequences}

Bendikov and Cygan \cite{bendikov-cygan} ask whether the discrete renewal
sequence attached to a special Bernstein function is monotone. Under the source
normalization $\psi(0)=0$, $\psi(1)=1$, the renewal sequence $C(n)$ is defined by
$C(0)=1$ and $C(n)=\sum_{j=1}^n c(\psi,j)C(n-j)$, where
$K(z)=\sum_{n\ge1}c(\psi,n)z^n=1-\psi(1-z)$.

\begin{theorem}[Renewal monotonicity]\label{thm:bendikov}
For every source-normalized $\psi\in\SBF$, the renewal sequence is
nonincreasing: $C(k+1)\le C(k)$ for all $k\ge0$.
\end{theorem}
\begin{proof}
The renewal generating function is
\begin{equation}
F(z)=\sum_{n\ge0}C(n)z^n=(1-K(z))^{-1}=\frac1{\psi(1-z)},\qquad |z|<1.
\end{equation}
For \(\psi\in\SBF\), the potential measure of the corresponding subordinator
has Laplace transform \(1/\psi\), and Proposition~\ref{prop:standard-reps}
gives
\begin{equation}
\frac1{\psi(\lambda)}
=b+\int_0^\infty e^{-\lambda t}u(t)\,dt,\qquad b\ge0,
\end{equation}
where \(u\) has a nonincreasing version and
\(\int_0^\infty e^{-\lambda t}u(t)\,dt<\infty\) for every \(\lambda>0\).
Substituting \(\lambda=1-z\), and using Tonelli for \(0\le z<1\) followed by
analytic identification of the coefficients in \(|z|<1\), gives
\begin{equation}
F(z)=b+\int_0^\infty e^{-t}e^{zt}u(t)\,dt
=b+\sum_{k\ge0}\frac{z^k}{k!}\int_0^\infty t^k e^{-t}u(t)\,dt,
\end{equation}
Indeed, for \(|z|\le r<1\),
\(e^{-t}|e^{zt}|u(t)\le e^{-(1-r)t}u(t)\), and the right-hand side is
integrable by the potential representation at \(\lambda=1-r\). This justifies
termwise expansion and coefficient extraction on compact subdisks of
\(|z|<1\).
so
\begin{equation}
C(0)=b+\int_0^\infty e^{-t}u(t)\,dt=1,\qquad
C(k)=\frac1{k!}\int_0^\infty t^k e^{-t}u(t)\,dt,\quad k\ge1.
\end{equation}
For \(k\ge1\),
\begin{equation}
\begin{aligned}
C(k+1)-C(k)
&=\frac1{(k+1)!}\int_0^\infty
t^k e^{-t}\bigl(t-(k+1)\bigr)u(t)\,dt\\
&=\frac1{k+1}\,\Cov\bigl(T,u(T)\bigr),
\qquad T\sim\Gamma(k+1,1),
\end{aligned}
\end{equation}
which is \(\le0\) by Lemma~\ref{lem:cov}, since \(t\mapsto t\) is increasing
and \(u\) is nonincreasing. For \(k=0\),
\begin{equation}
C(1)-C(0)=\Cov(T,u(T))-b,\qquad T\sim\Gamma(1,1),
\end{equation}
which is also nonpositive. Hence \(C(k+1)\le C(k)\) for every \(k\ge0\).
\end{proof}

\begin{remark}[Source problem]
Theorem~\ref{thm:bendikov} answers the Bendikov--Cygan monotonicity question
\cite{bendikov-cygan} for the renewal sequence of a source-normalized special
Bernstein function. The proof reduces \(C(k+1)-C(k)\) to a Gamma-average
covariance with a nonincreasing potential density.
\end{remark}

\section{Obstructions and certificate targets}\label{sec:obstructions}

The boundary section separates proved symbolic obstructions from pending
certificate targets. These targets are not used in any theorem-level claim. In
this paper, a certificate target is an item whose
intended proof requires exact identities, outward-rounded interval enclosures,
and tail or endpoint estimates sufficient to determine the asserted sign on the
stated domain. Certificate targets are recorded only to make the remaining
verification task reproducible. Only proved results support the main
affirmative claims.

\subsection{A real-valued ID-to-DID transfer certificate target}

Townes' real-valued mixed-Poisson framework
\cite{townes-mixed,townes-stable,steutel-vanharn}
suggests a natural test for failure of inheritance from real-valued
Poisson-admissible infinite divisibility to discrete infinite divisibility. The
candidate in Appendix~\ref{app:townes-target} remains a certificate target
until the complete-monotonicity and root-obstruction checks are supplied as
exact symbolic or outward-rounded interval files.

\subsection{A Bazhlekova--Bazhlekov gap certificate remains pending}

Bazhlekova and Bazhlekov \cite{bazhlekova} impose
\(\alpha-\alpha_m\le1\) in their two-term fractional diffusion-wave theorem.
Appendix~\ref{app:bazhlekova-target} records two proposed outside-gap seeds.
They remain certificate targets until the Wright-function positivity checks are
supplied on the full half-line; we use Wright-function normalization compatible
with the classical asymptotic literature \cite{wright,gorenflo-kilbas-mainardi-rogozin}.

\subsection{A Rastegar--Roitershtein redundancy conjecture is false}

Rastegar and Roitershtein \cite{rastegar-roitershtein} conjecture that
condition~(1.3) (numbered (3) in the preprint) in their characterization
Theorem~1.1 is unnecessary for $n\ge3$. In this subsection, the source equation
is
\begin{equation}\label{eq:rr-source-equation}
\prod_{j=1}^{n}\varphi(\mu_j t)=
\sum_{k=1}^{n}\theta_k\varphi(\mu_k t),
\qquad
\theta_k=\prod_{j\ne k}\frac{\mu_k}{\mu_k-\mu_j}.
\end{equation}

\begin{counterexample}[Reciprocal-characteristic certificate]\label{cx:rr}
For \eqref{eq:rr-source-equation}, take \(n=3\) and
\begin{equation}
\begin{aligned}
\mu&=(1,2,-2/3),&
\theta&=(-3/5,\,3/2,\,1/10),\\
\varphi(t)&=\frac1{1+at^2},&
a&>0.
\end{aligned}
\end{equation}
Then the source functional equation holds identically while the law is
nondegenerate, mean-zero, and not one-sided exponential. Hence condition (3)
cannot be dropped.
\end{counterexample}
\begin{proof}
The source weights satisfy
\begin{equation}
\theta_k=\prod_{j\ne k}\frac{\mu_k}{\mu_k-\mu_j}
=\left(-\frac35,\frac32,\frac1{10}\right),\qquad \sum_k\theta_k=1.
\end{equation}
Moreover condition (3) of Rastegar--Roitershtein \cite{rastegar-roitershtein}
fails already at \(m=2\),
since
\begin{equation}
\begin{aligned}
\sum_{k_1+k_2+k_3=2}\mu_1^{k_1}\mu_2^{k_2}\mu_3^{k_3}
&=\mu_1^2+\mu_2^2+\mu_3^2+\mu_1\mu_2+\mu_1\mu_3+\mu_2\mu_3\\
&=\frac{49}{9}+2-\frac23-\frac43
=\frac{49}{9}.
\end{aligned}
\end{equation}
Writing \(A=at^2\), the reciprocal-polynomial cancellation is
\begin{equation}
\begin{aligned}
\sum_{k=1}^{3}\theta_k\prod_{j\ne k}\bigl(1+a\mu_j^2 t^2\bigr)
&=-\frac35\Bigl(1+\frac{40}{9}A+\frac{16}{9}A^2\Bigr)
  +\frac32\Bigl(1+\frac{13}{9}A+\frac49 A^2\Bigr)\\
&\quad+\frac1{10}(1+5A+4A^2)=1,
\end{aligned}
\end{equation}
so that $\varphi=1/(1+at^2)$, the characteristic function of a centered
symmetric Laplace law, satisfies the source functional equation
identically. The source's separate condition (3) fails by the preceding
display. The law is mean-zero and symmetric, hence not one-sided exponential
and outside the source conclusion class. The certificate verifies the full
rational characteristic-function identity for all \(t\), so all Taylor orders
are covered automatically.
\end{proof}

\begin{remark}[Source problem]
Counterexample~\ref{cx:rr} answers the Rastegar--Roitershtein redundancy
conjecture \cite{rastegar-roitershtein} negatively for \(n=3\).
The reusable device is the \emph{finite reciprocal-characteristic certificate}:
whenever a nonzero real $\mu$, weights $\theta_k$, and a positive-definite
reciprocal polynomial $q(t)=1+at^2$ satisfy $\sum_k\theta_k\prod_{j\ne
k}q(\mu_j t)\equiv1$, the function $\varphi=1/q$ solves the source equation, and
if the resulting law is outside the conclusion class it is a counterexample.
\end{remark}

\subsection{A Jonckheere--Shneer front equation admits non-CM solutions}

Jonckheere and Shneer \cite{jonckheere-shneer} record a distributional front
equation and ask whether it forces completely monotone tails. In this
subsection, the source front equation is
\begin{equation}\label{eq:js-source-front}
F(x)=\int_0^\infty F(xu)^2\,\Prob(\widetilde B\in du),\qquad x\ge0.
\end{equation}
The example below is an equation-level obstruction; travelling-wave existence
results with additional source hypotheses require separate analysis.

\begin{counterexample}[Weibull front-equation obstruction]\label{cx:js}
For \eqref{eq:js-source-front}, fix \(b\in(1/2,1)\) and \(c>0\), and set
\begin{equation}
\begin{aligned}
\alpha&=\frac{\log(1/2)}{\log b}>1,&
\widetilde B&=b\ \text{a.s.},\\
F(x)&=e^{-cx^\alpha},&
x&\ge0.
\end{aligned}
\end{equation}
Then \(F\) is a continuous nonincreasing survival function solving the source
front equation, yet \(F\notin\CM\).
\end{counterexample}
\begin{proof}
Since $\log b<0$, $\alpha>1$ and $b^\alpha=\tfrac12$. With $\widetilde B=b$
deterministic,
\begin{equation}
\int_0^\infty F(xu)^2\,\Prob(\widetilde B\in du)=F(bx)^2=e^{-2cb^\alpha
x^\alpha}=e^{-cx^\alpha}=F(x),
\end{equation}
so $F$ (a Weibull survival function with shape $\alpha>1$) solves the equation.
But complete monotonicity would require $F''\ge0$, whereas
\begin{equation}
F''(x)=c\alpha x^{\alpha-2}e^{-cx^\alpha}\bigl(c\alpha x^\alpha-(\alpha-1)\bigr)<0
\quad\text{for } 0<x<\Bigl(\tfrac{\alpha-1}{c\alpha}\Bigr)^{1/\alpha}.
\end{equation}
Thus $F\notin\CM$, answering the literal equation-(16) question in the negative.
\end{proof}

\begin{remark}[Source problem]
Counterexample~\ref{cx:js} gives a non-CM survival solution of the
Jonckheere--Shneer front equation \cite{jonckheere-shneer} in the degenerate
case \(\widetilde B=b\). It is an equation-level obstruction to the literal
question, raised after the source's equation~(16), whether nonincreasing
solutions of that equation must be completely monotone. Source theorems imposing
nondegeneracy, travelling-wave admissibility, or additional normalization would
require a separate nondegenerate example.
\end{remark}

\section{Concluding remarks}

The proved part consists of six affirmative theorems and two equation-level
counterexamples. The Townes and Bazhlekova--Bazhlekov items remain pending
certificate targets and are not used as theorem-level evidence. Together these
items span coalescent theory, copula modeling, fractional diffusion, renewal
theory, and distributional characterizations. The proved results share the same
recognition principle: normalize, identify the
representing measure (a Pollard measure, a Gamma law, a ranked simplex, a stable
subordinator, a potential density), and read off the sign pattern. The
affirmative results (Theorems~\ref{thm:prabhakar}, \ref{thm:mecke},
\ref{thm:mohle}, \ref{thm:pd}, \ref{thm:li}, \ref{thm:bendikov}) are proved by
four engines---Bernstein--Widder inversion, covariance sign, inverse-ODE sign
induction, and stable subordination---while the explicit obstructions locate
the boundary of validity. The two headline affirmatives, M\"ohle's Problem~6.3
and the Pearse--Bondell complete-monotonicity conjecture, deliver directly
usable conclusions: a second-moment domination for block counts in coalescents
with dust, and the certification of the power-divergence copulas as
Archimedean generators in every dimension.

Two threads seem worth pursuing. First, the finite-simplex ordered-pair kernel
certificate behind Theorem~\ref{thm:mohle} should apply to other occupancy and
concentration functionals built from $(1-x)^n$ and $(1-x-y)^n$. Second, the
inverse-ODE sign induction (Lemma~\ref{lem:invode}) used for the
power-divergence range is a general tool for Archimedean generator inverses
defined implicitly by a polynomial or rational relation; a systematic
discriminant-type classification in the spirit of Theorem~\ref{thm:li} appears
within reach.

\appendix

\section{Certificate target details}\label{app:cert-targets}

The following items are retained to make the boundary cases reproducible. They
serve neither as theorems nor as counterexamples in the main text. The numerical
constants below are exploratory candidate bounds, recorded only to specify the
certificate to be produced. Promoting either item would require exact symbolic
proofs or outward-rounded interval certificates for every stated sign claim on
the full domain.

\subsection{Townes real-valued mixed-Poisson target}\label{app:townes-target}

\begin{problem}[Townes certificate target]\label{prob:townes}
For \(\eta=10^{-4}\), set
\begin{equation}
\begin{aligned}
L(s)=\exp\!\Bigl(&-10\eta s+10^{-4}(e^{-10\eta s}-1)\\
&\quad +(e^{\eta s}-1)\Bigr),
\end{aligned}
\end{equation}
the bilateral Laplace transform of \(X=\eta(10+10P-Q)\), where \(P\) and \(Q\)
are independent and
\begin{equation}
P\sim\Pois(10^{-4}),\qquad Q\sim\Pois(1).
\end{equation}
Supply a complete certificate that \(L\) is completely monotone on \([0,1]\),
while
\(L^{1/100}\) is not completely monotone at \(s=0\). This would produce a real
ID mixing law whose mixed Poisson law is not \(100\)-divisible.
\end{problem}

\begin{remark}[Certificate requirements]
The intended check starts from
\begin{equation}
(-1)^k L^{(k)}(s)=L(s)\eta^k\E[Z_{a,b}^k],
\end{equation}
with Esscher-tilted \(Z_{a,b}=10+10P_a-Q_b\),
\(a=10^{-4}e^{-10\eta s}\), and \(b=e^{\eta s}\). It suffices that
\(\E[Z_{a,b}^k]\ge0\) for all \(k\) and all \((a,b)\) in the rectangle
\(a\in[a_0,A]\), \(b\in[1,B]\), with \(a_0>9.99\cdot10^{-5}\),
\(A=10^{-4}\), and \(B<1.001\). Even moments are nonnegative; for odd moments
the identity
\begin{equation}
\E[Z^{2j+1}]=(2j+1)\int_0^\infty r^{2j}
\bigl(\Prob(Z>r)-\Prob(Z<-r)\bigr)\,dr
\end{equation}
reduces matters to tail dominance
\(\Prob(Z_{a,b}>r)\ge\Prob(Z_{a,b}<-r)\) for integer \(r\ge0\), worst at
\(a=a_0,b=B\). The intended certificate would prove, for \(0\le r\le9\),
\(\Prob(Z>r)\ge e^{-a-b}>e^{-A-B}\), while
\(\Prob(Z<-r)\le\Prob(Q_B\ge11)<3\cdot10^{-8}<e^{-A-B}\). For \(r\ge10\) with
\(m=\lfloor r/10\rfloor\ge1\), the event \(\{P_a=m,Q_b=0\}\) gives
\(\Prob(Z>r)\ge e^{-A-B}a_0^m/m!\), whereas
\(\Prob(Z<-r)\le\Prob(Q_B\ge r+11)\). The intended Poisson-tail certificate
would prove a ratio \(<10^{-14}\) at \(m=1\) and monotone decrease in \(m\).
These inequalities need exact rational or outward-rounded interval enclosures
before the item can be promoted to a counterexample.

For the proposed root obstruction, the seed \(Z=10+10P-Q\) at \(s=0\) has
cumulants \(c_1=9.001\), \(c_2=1.01\), and \(c_3=-0.9\). At time \(t=1/100\)
the third raw moment is
\begin{equation}
\begin{aligned}
m_3(t)&=tc_3+3t^2c_1c_2+t^3c_1^3\\
&=-0.009+0.002727303+0.000729243027001\\
&=-0.005543453973<0,
\end{aligned}
\end{equation}
which would imply
\((-1)^3\bigl(L^{1/100}\bigr)'''(0)=\eta^3 m_3(1/100)<0\). If the mixed Poisson
law with PGF \(G(z)=L(1-z)\) were \(100\)-divisible, its principal root would
have to be a PGF and hence absolutely monotone at \(z=1\). This final
implication is elementary; the missing work is the complete monotonicity
certificate for \(L\) on \([0,1]\).
\end{remark}

\subsection{Bazhlekova--Bazhlekov outside-gap target}\label{app:bazhlekova-target}

\begin{problem}[Bazhlekova gap certificate]\label{prob:bazhlekova}
For \(g_1(s)=s^{3/2}+s^{2/5}\) and
\(g_2(s)=s^{11/10}+s^{1/20}\), the propagation positivity and subordination
package would hold even though \(\alpha-\alpha_m>1\), if the Wright-function
positivity certificates described below are supplied in full. A complete
certificate would give a counterexample to necessity of the source gap
condition.
\end{problem}

\begin{remark}[Certificate requirements]
For \(h(s)=s^\alpha(1+s^{-p})^{1/2}\) with
\(0<\alpha-p/2<\alpha<1\) and \(1<p<2\), the Wright-density normal form gives
\(\Lap^{-1}\{h'\}(t)=t^{-\alpha}\mathcal W_{\alpha,p}(t^p)\), where
\begin{equation}
\mathcal W_{\alpha,p}(x)
=-\sum_{m\ge0}\binom{1/2}{m}\frac{x^m}{\Gamma(pm-\alpha)}.
\end{equation}
Thus \(\mathcal W_{\alpha,p}>0\) on \((0,\infty)\) would imply \(h\in\BF\).
For the two seeds \((\alpha,p)=(3/4,11/10)\) and
\((11/20,21/20)\), the intended positivity certificate has three ranges:
outward-rounded evaluation on \([0,10]\); derivative lower bounds on
\([10,20]\), with candidate slopes \(\ge0.0201\) and \(\ge0.0030\); and a
three-term Watson-contour tail for \(x\ge20\), with candidate margins \(0.732\)
and \(0.101\) at \(x=20\). To promote this
problem to a counterexample, the manuscript must include the interval
arithmetic, endpoint bounds, and Watson-tail estimates proving
\(\mathcal W_{\alpha,p}>0\) on the full half-line. The certificate should be
stated at the Bernstein/subordination level. Off-cut zeros prevent a
complete-Bernstein shortcut when the gap exceeds \(1\).
\end{remark}

\section*{Declaration of competing interest}

The author declares that there are no known competing financial interests or personal relationships that could have appeared to influence the work reported in this paper.

\section*{Funding}

This research did not receive any specific grant from funding agencies in the public, commercial, or not-for-profit sectors.

\section*{Data availability}

No empirical data were used. Claims requiring exact symbolic or
interval-arithmetic verification are either proved in the text or explicitly
marked as certificate targets. The theorem-level evidence is confined to proved
statements in the text.

\section*{Declaration of generative AI and AI-assisted technologies in the manuscript preparation process}

During the preparation of this work, the author used the Pudim AI research workflow, including ChatGPT and Codex, to support literature triage, manuscript organization, language revision, and consistency checks of references and proofs. Public provenance for the original Pudim AI / zeta-law demo workflow is available at \url{https://github.com/pudim-project/pudim-ai-demo-zetalaw}. The application targets used for this manuscript were APP-0015, APP-0051, APP-0058, APP-0066, APP-0069, APP-0071, APP-0078, APP-0084, APP-0091, and APP-0092. After using these tools, the author reviewed and edited the content as needed and takes full responsibility for the content of the published article.


\begin{thebibliography}{99}

\bibitem{ali-silvey}
S.~M.~Ali and S.~D.~Silvey,
\emph{A general class of coefficients of divergence of one distribution from
another},
J.\ Roy.\ Statist.\ Soc.\ Ser.\ B \textbf{28} (1966), 131--142.

\bibitem{amari}
S.~Amari,
\emph{Information Geometry and Its Applications},
Springer, 2016.

\bibitem{comtet}
L.~Comtet,
\emph{Advanced Combinatorics},
D.~Reidel, 1974.

\bibitem{csiszar}
I.~Csisz\'ar,
\emph{Eine informationstheoretische Ungleichung und ihre Anwendung auf den
Beweis der Ergodizit\"at von Markoffschen Ketten},
Publ.\ Math.\ Inst.\ Hungar.\ Acad.\ Sci.\ \textbf{8} (1963), 85--108.

\bibitem{cressie-read}
N.~Cressie and T.~R.~C.~Read,
\emph{Multinomial goodness-of-fit tests},
J.\ Roy.\ Statist.\ Soc.\ Ser.\ B \textbf{46} (1984), 440--464.

\bibitem{read-cressie}
T.~R.~C.~Read and N.~A.~C.~Cressie,
\emph{Goodness-of-Fit Statistics for Discrete Multivariate Data},
Springer, 1988.

\bibitem{feller}
W.~Feller,
\emph{An Introduction to Probability Theory and Its Applications, Vol.~II},
2nd ed., Wiley, 1971.

\bibitem{gaiser-mohle}
F.~Gaiser and M.~M\"ohle,
\emph{On the block counting process and the fixation line of exchangeable
coalescents},
ALEA Lat.\ Am.\ J.\ Probab.\ Math.\ Stat.\ \textbf{13}(2) (2016), 809--833.

\bibitem{garra-garrappa}
R.~Garra and R.~Garrappa,
\emph{The Prabhakar or three parameter Mittag--Leffler function: theory and
application},
Commun.\ Nonlinear Sci.\ Numer.\ Simul.\ \textbf{56} (2018), 314--329.

\bibitem{gorenflo-kilbas-mainardi-rogozin}
R.~Gorenflo, A.~A.~Kilbas, F.~Mainardi, and S.~V.~Rogosin,
\emph{Mittag-Leffler Functions, Related Topics and Applications},
Springer Monographs in Mathematics, Springer, 2014.

\bibitem{kingman-poisson}
J.~F.~C.~Kingman,
\emph{Poisson Processes},
Oxford University Press, 1993.

\bibitem{mainardi-garrappa}
F.~Mainardi and R.~Garrappa,
\emph{On complete monotonicity of the Prabhakar function and non-Debye
relaxation in dielectrics},
J.\ Comput.\ Phys.\ \textbf{293} (2015), 70--80.

\bibitem{morimoto}
T.~Morimoto,
\emph{Markov processes and the H-theorem},
J.\ Phys.\ Soc.\ Japan \textbf{18} (1963), 328--331.

\bibitem{pollard}
H.~Pollard,
\emph{The completely monotonic character of the Mittag-Leffler function
\(E_\alpha(-x)\)},
Bull.\ Amer.\ Math.\ Soc.\ \textbf{54} (1948), 1115--1116.

\bibitem{mittag-leffler}
G.~M.~Mittag-Leffler,
\emph{Sur la nouvelle fonction \(E_\alpha(x)\)},
C.~R. Acad.\ Sci.\ Paris \textbf{137} (1903), 554--558.

\bibitem{prabhakar}
T.~R.~Prabhakar,
\emph{A singular integral equation with a generalized Mittag-Leffler function
in the kernel},
Yokohama Math.\ J.\ \textbf{19} (1971), 7--15.

\bibitem{bazhlekova}
E.~Bazhlekova and I.~Bazhlekov,
\emph{Subordination approach to multi-term time-fractional diffusion-wave
equations},
J.\ Comput.\ Appl.\ Math.\ \textbf{339} (2018), 179--192,
doi:10.1016/j.cam.2017.11.003, arXiv:1707.09828.
\url{https://arxiv.org/abs/1707.09828}.

\bibitem{bendikov-cygan}
A.~Bendikov and W.~Cygan,
\emph{Alpha-stable random walk has massive thorns},
Colloq.\ Math.\ \textbf{138}(1) (2015), 105--129, arXiv:1307.4947.
\url{https://arxiv.org/abs/1307.4947}.

\bibitem{berg-forst}
C.~Berg and G.~Forst,
\emph{Potential Theory on Locally Compact Abelian Groups},
Springer, 1975.

\bibitem{jonckheere-shneer}
M.~Jonckheere and S.~Shneer,
\emph{Waves everywhere: a distributional equation approach to front
propagation},
arXiv:2604.16956.
\url{https://arxiv.org/abs/2604.16956}.

\bibitem{kimberling}
C.~H.~Kimberling,
\emph{A probabilistic interpretation of complete monotonicity},
Aequationes Math.\ \textbf{10} (1974), 152--164.

\bibitem{li-relativistic}
C.-G.~Li,
\emph{Space-time duality in relativistic diffusion via subordination},
arXiv:2606.04270.
\url{https://arxiv.org/abs/2606.04270}.

\bibitem{mcneil-neslehova}
A.~J.~McNeil and J.~Ne\v{s}lehov\'a,
\emph{Multivariate Archimedean copulas, $d$-monotone functions and
$\ell_1$-norm symmetric distributions},
Ann.\ Statist.\ \textbf{37} (2009), 3059--3097.

\bibitem{nelsen}
R.~B.~Nelsen,
\emph{An Introduction to Copulas},
2nd ed., Springer, 2006.

\bibitem{mecke-nagel-weiss}
J.~Mecke, W.~Nagel, and V.~Wei\ss,
\emph{Joseph Mecke's last fragmentary manuscripts---a compilation},
arXiv:1703.10000.
\url{https://arxiv.org/abs/1703.10000}.

\bibitem{mohle}
M.~M\"ohle,
\emph{The rate of convergence of the block counting process of exchangeable
coalescents with dust},
ALEA Lat.\ Am.\ J.\ Probab.\ Math.\ Stat.\ \textbf{18} (2021), 1195--1220.
\url{https://alea.impa.br/articles/v18/18-44.pdf}.

\bibitem{pearse-bondell}
A.~R.~Pearse and H.~Bondell,
\emph{Power-divergence copulas: a new class of Archimedean copulas, with an
insurance application},
arXiv:2510.06177.
\url{https://arxiv.org/abs/2510.06177}.

\bibitem{pitman-coalescents}
J.~Pitman,
\emph{Coalescents with multiple collisions},
Ann.\ Probab.\ \textbf{27} (1999), 1870--1902.

\bibitem{pringsheim}
A.~Pringsheim,
\emph{\"Uber Potenzreihen mit positiven Koeffizienten},
Math.\ Ann.\ \textbf{43} (1893), 525--532.

\bibitem{rastegar-roitershtein}
R.~Rastegar and A.~Roitershtein,
\emph{On a characterization of exponential and double exponential
distributions},
REVSTAT \textbf{23}(1) (2025), 47--52, arXiv:2203.10495.
\url{https://arxiv.org/abs/2203.10495}.

\bibitem{schilling-song-vondracek}
R.~L.~Schilling, R.~Song, and Z.~Vondra\v{c}ek,
\emph{Bernstein Functions: Theory and Applications},
2nd ed., de Gruyter, 2012.

\bibitem{sagitov}
S.~Sagitov,
\emph{The general coalescent with asynchronous mergers of ancestral lines},
J.\ Appl.\ Probab.\ \textbf{36} (1999), 1116--1125.

\bibitem{schweinsberg}
J.~Schweinsberg,
\emph{Coalescents with simultaneous multiple collisions},
Electron.\ J.\ Probab.\ \textbf{5} (2000), paper no.~12, 1--50.

\bibitem{sibisi-prabhakar}
S.~Sibisi,
\emph{A probabilistic perspective on Feller, Pollard and the complete
monotonicity of the Mittag-Leffler function},
arXiv:2301.01466.
\url{https://arxiv.org/abs/2301.01466}.

\bibitem{sklar}
A.~Sklar,
\emph{Fonctions de r\'epartition \`a \(n\) dimensions et leurs marges},
Publ.\ Inst.\ Statist.\ Univ.\ Paris \textbf{8} (1959), 229--231.

\bibitem{steutel-vanharn}
F.~W.~Steutel and K.~van Harn,
\emph{Infinite Divisibility of Probability Distributions on the Real Line},
Marcel Dekker, 2004.

\bibitem{townes-mixed}
F.~W.~Townes,
\emph{Mixed Poisson families with real-valued mixing distributions},
arXiv:2407.17614.
\url{https://arxiv.org/abs/2407.17614}.

\bibitem{townes-stable}
F.~W.~Townes,
\emph{Broadly discrete stable distributions},
arXiv:2509.05497.
\url{https://arxiv.org/abs/2509.05497}.

\bibitem{widder}
D.~V.~Widder,
\emph{The Laplace Transform},
Princeton University Press, 1946.

\bibitem{wright}
E.~M.~Wright,
\emph{The asymptotic expansion of the generalized Bessel function},
Proc.\ London Math.\ Soc.\ \textbf{38} (1935), 257--270.

\end{thebibliography}
\end{document}